\documentclass[review]{elsarticle}
\usepackage[top=2.5cm, bottom=2.5cm, left=2.5cm, right=2cm]{geometry}
\usepackage{amssymb,latexsym,amscd,amsmath,amsfonts,enumerate,supertabular,amsthm}
\usepackage{graphicx}
\usepackage{color}
\usepackage{tabularx}
\usepackage{subfigure}
\usepackage{enumerate}
\usepackage{epstopdf}
\usepackage{natbib,float}
\usepackage{lscape}
\usepackage{lineno,hyperref}
\newtheorem{theorem}{Theorem}
\newtheorem{lemma}{Lemma}

\newtheorem{proposition}{Proposition}

\newtheorem{remark}{Remark}

\newtheorem{definition}{Definition}
\modulolinenumbers[5]

\journal{Elsevier}

\begin{document}

\begin{frontmatter}

\title{A novel second-order nonstandard finite difference method for solving one-dimensional autonomous dynamical systems}

%% or include affiliations in footnotes:
%\author[mymainaddress,mysecondaryaddress]{Elsevier Inc}
%\ead[url]{www.elsevier.com}

\author[Manh Tuan Hoang]{Manh Tuan Hoang\corref{mycorrespondingauthor}}
\cortext[mycorrespondingauthor]{Corresponding author}
\ead{tuanhm14@fe.edu.vn; hmtuan01121990@gmail.com}
\address[mymainaddress]{Department of Mathematics, FPT University, Hoa Lac Hi-Tech Park, Km29 Thang Long Blvd, Hanoi, Viet Nam}
\begin{abstract}
{
In this work, a novel second-order nonstandard finite difference (NSFD) method that preserves simultaneously the positivity and local asymptotic stability of one-dimensional autonomous dynamical systems is introduced and analyzed. This method is based on novel non-local approximations for right-hand side functions of differential equations in combination with nonstandard denominator functions. The obtained results not only resolve the contradiction between the dynamic consistency and high-order accuracy of NSFD methods but also improve and extend some well-known results that have been published recently in [Applied Mathematics Letters 112(2021) 106775], [AIP Conference Proceedings 2302(2020) 110003] and [Applied Mathematics Letters 50(2015) 78-82]. Furthermore, as a simple  but important application, we apply the constructed NSFD method for solving the logistic, sine, cubic, and Monod equations; consequently, the NSFD schemes constructed in the earlier work  [Journal of Computational and Applied Mathematics 110(1999) 181-185] are improved significantly. Finally, we report some numerical experiments to support and illustrate the theoretical assertions as well as advantages of the constructed NSFD method.
} 
\end{abstract}
%%%%%%%%%%%%%%%%%%%%%%%%%%%%%%%%%%%%%%%
\begin{keyword}
NSFD method \sep Second-order \sep  Positivity \sep Elementary stable \sep Autonomous dynamical system 
\MSC[2010] 	37M05, 37M15,		65L05, 	65P99
\end{keyword}
\end{frontmatter}
%\linenumbers
\section{Introduction}
The concept of nonstandard finite difference (NSFD) schemes was first introduced by Mickens in 1980 to overcome drawbacks of standard finite difference ones \cite{Anguelov, Mickens1, Mickens2, Mickens3, Mickens4, Mickens5, Patidar1, Patidar2}. Nowadays, NSFD schemes have been widely used as one of the powerful and efficient approaches for solving differential equation models arising in both theory and applications. The main advantage of NSFD schemes over standard ones is that they can  preserve important qualitative properties of solutions of the corresponding differential equations regardless of chosen step sizes.  However, most of NSFD schemes are only of first order accuracy. For this reason, there have been many attempts to formulate higher-order NSFD schemes for ordinary differential equations (ODEs) \cite{Chen-Charpentier, Gupta, DH1, Kojouharov, MVaquero1, MVaquero2}.\par
%%%%%%%%%%%%%%%%%%%%%%%%%%%%%%%%%%%%%%%
%%%%%%%%%%%%%%%%%%%%%%%%%%%%%%%%%%%%%%%%
In very recent works \cite{Gupta, Kojouharov}, second-order NSFD methods preserving the local asymptotic stability (LAS) of one-dimensional autonomous ODEs have been introduced. These results partially resolved the contradiction between the dynamic consistency and high-order accuracy of NSFD schemes; however, the constructed methods did not preserve the positivity of solutions of differential equations.\par
%%%%
%%
Motivated and inspired by this, in this work we construct a novel NSFD method that not only preserves simultaneously the positivity and LAS  of one-dimensional autonomous dynamical systems but also is accurate of order $2$. The method is based on novel non-local approximations for right-hand side functions of differential equations in combination with nonstandard denominator functions. The obtained results not only resolve the contradiction between the dynamic consistency and high-order accuracy of NSFD schemes but also improve and extend the well-known results constructed in \cite{Gupta, Kojouharov, Wood}. { Moreover, it should be emphasized that elementary stable NSFD schemes  introduced in some previous works \cite{Anguelov1, Dimitrov, Dimitrov1}  are only of first order accuracy and not positive}.\par
%%%%%%%%%%%%%%%%%%%%%%%%%%%%%
{
It is worthy noting that our constructed NSFD method is completely different from NSFD schemes formulated in \cite{Wood} and \cite{Mickens1, Mickens2, Mickens3, Mickens4, Mickens5}. More clearly, we use non-local approximations based on novel representations of right-hand side functions and especially,  denominator functions depend not only on the step size but also on numerical approximations, meanwhile, the NSFD schemes introduced in \cite{Mickens1, Mickens2, Mickens3, Mickens4, Mickens5} and \cite{Wood} used denominator functions only depend on the step size. Moreover, the NSFD scheme formulated in \cite{Wood} used a non-local approximation based on the sign of the right-hand side functions at each iterative step; however, our method does not require the determination of the sign of the right-hand functions. In fact, the positive and elementary stable  NSFD method in \cite{Wood} can be considered as a particular case of our method, but it is only of first-order accuracy. This will be confirmed by a numerical example performed in Subsection \ref{subsec4.1}.\par
}
%%%%%%%%%%%%%%%%%%%%%%%%%%%%%%%%%%
Importantly, the present idea used to construct the NSFD method can be developed for formulating dynamically consistent NSFD methods having high-order accuracy for general multi-dimensional autonomous and non-autonomous dynamical systems, which are governed by ODEs, fractional-order ODEs and partial differential equations. Furthermore, the constructed method is fully explicit; hence, it is easy to be implemented and can be combined with extrapolation techniques to improve the accuracy.  \par
%%%%%%%%%%%%%%%%%%%%%%%%%%%%%%
{ As a simple but important application, we apply the constructed NSFD method for solving  the logistic, sine, cubic, and Monod equations, which were considered in \cite{Mickens6}, and general multi-dimensional dynamical systems. The result is that we obtain positive and elementary stable NSFD schemes of seconder-order accuracy for these equations. Meanwhile, the NSFD schemes constructed in \cite{Mickens6} are only accurate of order $1$. Therefore, the results constructed in \cite{Mickens6} are improved significantly.}\par
%%%%%%%%%%%%%%%%%%%%%%%%%%%%%%%%%%%%%%%%%%%%%%%%%%%%%%%%%%%%%%%%%%%%%%%%%%%%
This work is organized as follows. In Section \ref{sec2}, we provide some definitions and preliminaries. The NSFD method is constructed and analyzed in Section \ref{sec3}. Section \ref{sec4} reports a set of numerical simulations and some applications of the NSFD method. The last section presents some remarks and conclusions.
%%%%%%%%%%%%%
\section{Definitions and preliminaries}\label{sec2}
Consider general autonomous ODEs of the form 
\begin{equation}\label{eq:1}
\dfrac{dy}{dt} = f(y), \quad t \geq 0, \quad y(0) = y_0,
\end{equation}
where $y: \mathbb{R} \to \mathbb{R}$, $f \in C^2(\mathbb{R})$ and $y_0 \in \mathbb{R}$. Besides, we consider a general finite difference method for Eq. \eqref{eq:1}
\begin{equation}\label{eq:1a}
D_h(y_n) = F_h(f; y_n),
\end{equation}
where $D_h(y_n) \approx \dfrac{dy}{dt}|_{t = t_n}$, $F_h(f; y_n) \approx f(y(t_n))$, $y_n \approx y(t_n)$, $t_n = nh$ and $h > 0$ is the step size.
{
%%%%%%%%%%%%%%%
\begin{definition}[Positivity-preserving numerical schemes \cite{Mickens1, Mickens2, Mickens3, Mickens4, Mickens5, Wood}]
The finite difference method \eqref{eq:1a} is called positive, if, for any value of the step size $h$, and
$y_0 \in \mathbb{R}^+ := \{y \in \mathbb{R}|\,  y\geq 0\}$ its solution remains positive, i.e., $y_n \in \mathbb{R}^+$ for all $n \in \mathbb{N}$.
\end{definition}
%%%%%%%%%%%%%%%%%%%%%%%%%%%%%%%%%%%%%
\begin{definition}[Elementary stable numerical schemes \cite{Anguelov, Dimitrov}]
The finite difference method \eqref{eq:1a} is said to be elementary stable if, for any value of
the step-size $h$, its only fixed points $y^*$ are the same as the equilibria of the ODE \eqref{eq:1} and the local stability properties of each $y^*$ are the same for both the differential equation and the difference method.
\end{definition}
%%%%%%%%%%%%%%%%%
\begin{definition}[Concept of NSFD schemes \cite{Anguelov, Mickens1, Mickens2, Mickens3, Mickens4, Mickens5}]
The one-step finite difference scheme \eqref{eq:1a} for solving Eq. \eqref{eq:1} is called a nonstandard
finite difference (NSFD) method  if at least one of the following conditions is satisfied:
\begin{itemize}
\item $D_h(y_n) = \dfrac{y_{n + 1} - y_n}{\varphi(h)}$, where $\varphi(h) = h + \mathcal{O}(h^2)$   is a non-negative function;
\item $F_h(f; y_n) = g(y_k, y_{k + 1}, h)$, where $g(y_k, y_{k+1}, h)$ is a non-local approximation of the right-hand side of Eq. \eqref{eq:1}.
\end{itemize}
\end{definition}
}
%%%%%%%%%%%
We recall that Eq. \eqref{eq:1} is said to be \textit{positive} if $y(t) \geq 0$ holds for all $t \geq 0$ whenever $y_0 \geq 0$. Thanks to \cite[Lemma 1]{Horvath}, we obtain the following result on the positivity of solutions of Eq. \eqref{eq:1}.
\begin{lemma}\label{plemma}
The set $\mathbb{R}^+$ is a positively invariant set of Eq. \eqref{eq:1} (i.e, $y(t) \geq 0$ for all $t \geq 0$ whenever $y_0 \geq 0$) if and only if $f(0) \geq 0$.
\end{lemma}
%%%%%%%%%%%%%%
%%%%%%%%%%%%%%%%%%%%%%
%%%%%%%%%%%%%%%%%%%%%%%
\section{Construction of the nonstandard finite difference method}\label{sec3}
\subsection{Auxiliary results}
We first introduce some notations
%%%%%%%%%%%%%%%
\begin{equation*}
\begin{split}
&\mathcal{F}^+ = \big\{f: \mathbb{R}_+ \to \mathbb{R}\big|\,f(y) \geq 0 \quad \text{for all} \quad y \geq 0\big\},\\
&\mathcal{F}^- = \big\{f: \mathbb{R}_+ \to \mathbb{R}\big|\,f(y) \leq 0 \quad \text{for all} \quad y \geq 0\big\},\\
&\mathcal{B}^L = \big\{f: \mathbb{R}_+ \to \mathbb{R}\big|\, f(y) \geq L \quad \text{for all}\quad y \geq 0\big\} \quad \text{for each} \quad M \geq 0.
\end{split}
\end{equation*}
It is easy to see that $\mathcal{B}^L $ is not empty for any $L \geq 0$, for example, $f(y) = L$, $f(y) = y^2 + L$ and $f(y) = Le^{y}$ belong to $\mathcal{B}^L$. Moreover, the set $\mathcal{B}^L \cap C^k(\mathbb{R}^+)$ ($k \geq 1$) is also not empty.
%%%%%%%%%%%%%%%%%%%%%%%%%%%%%%%%%%%%%%%%%%%%%%%%%%%%%%%
\begin{lemma}\label{Lemma1}
Let $f: \mathbb{R}^+ \to \mathbb{R}$ be a function  satisfying:
\begin{enumerate}[(i)]
\item $f(y) \in C^k(\mathbb{R}^+)$, \quad $k \geq 1$.
\item $f$ has a finite number of zeros, i.e, the equation $f(y) = 0$ has only a finite number of solutions.
\end{enumerate}
Then $f$ can be represented in the form
\begin{equation*}
f(y) = f^+(y) + f^{-}(y), \quad f^+, f^- \in C^k(\mathbb{R}^+),  \quad f^+ \in \mathcal{F}^+, \quad f^- \in \mathcal{F}^-.
\end{equation*}
\end{lemma}
%%%%%%%%%%%%%%%%%%%%%%%%%%
%%%%%%%%%%%%%%%%%%%%%%%%%%%
\begin{proof}
Assume that $0 \leq y_1 \leq y_2 < \ldots \leq y_m$ with $m < \infty$ are zeros of $f$. We set
\begin{equation*}
l = \min_{y \in [0, y_m]}f(y), \quad L = \max_{y \in [0, y_m]}f(y), \quad M = max\{|l|,\,\,|L|\}.
\end{equation*}
Since $f$ has only a finite number of zeros, $f(y)$ does not change sign when $y > y_m$. We consider two following cases for the sign of $f(y)$ when $y > y_m$.\\
\textbf{Case 1:} The case when $f(y) > 0$ for all $y > y_m$. In this case, we take
\begin{equation*}
f^+(y) = f(y) + g(y), \qquad f^-(y) = -g(y),  \quad \text{for some} \quad g \in \mathcal{B}^M \cap C^k(\mathbb{R}^+).
\end{equation*}
Obviously, $f^-(y) \leq 0$ for all $y \geq 0$. On the other hand, since $g \in B^M$, $g(y) \geq M$ for all $y \geq 0$. Consequently, for $y \in [0, y_m]$
\begin{equation*}
f^+(y) = f(y) + g(y) \geq l + M \geq l + |l| \geq 0.
\end{equation*}
Hence, $f^+(y) \geq 0$ for all $y \geq 0$.\\
%%%%%%%%%%%%%%%%%%%%%%%%%%%%%%%%%%%
%%%%%%%%%%%%%%%%%%%%%%%%%%%%%%%%%%%
\textbf{Case 2:} The case when $f(y) < 0$ for all $y > y_m$. In this case, we choose
\begin{equation*}
f^-(y) = f(y) - g(y), \qquad f^+(y) = g(y),  \quad \text{for some} \quad g \in \mathcal{B}^M \cap C^k(\mathbb{R}^+).
\end{equation*}
Obviously, $f^+(y) \geq 0$ for all $y \geq 0$. Furthermore
\begin{equation*}
f^-(y) = f(y) - f(y) \leq L - g(y) \leq L - M \leq L - |L| = 0 \quad \text{for all} \quad y \in [0, y_m].
\end{equation*}
Hence, $f^-(y) \geq 0$ for all $y \geq 0$.\\
The proof is completed by combining Cases 1 and 2.
\end{proof}
%%%%%%%%%%%%%%%%%%%%%%%%%%
%%%%%%%%%%%%%%%%%%%%%%%%%%%%%%%%%%%%%%%%%%%%%
{
\begin{remark}\label{remark1a}
Based on the proof of Lemma \ref{Lemma1}, we can write a procedure to determine representations of the function $f$ as follows:\\
Step 1. Determine all zeros the function $f$, i.e., solve the equation $f(y) = 0$ (using numerical methods such as Newton's method if needed). Since $f$ has only  only a finite number of zeros, this step is always done. Suppose that $0 \leq y_1 \leq y_2 \leq \ldots \leq y_m$ with $m < \infty$ are all the zeros of $f$.\\
Step 2. Compute (by using numerical methods if needed)
\begin{equation*}
l = \min_{y \in [0, y_m]}f(y), \quad L = \max_{y \in [0, y_m]}f(y), \quad M = max\{|l|,\,\,|L|\}.
\end{equation*}
Step 3. Determine the sign (positive or negative) of the function $f$ on $(y_m, \infty)$.  This step is always done because $f$ has only  only a finite number of zeros.\\
Step 4. If $f(y) > 0$ for all $y > y_m$, we take
\begin{equation*}
f^+(y) = f(y) + g(y), \qquad f^-(y) = -g(y),  \quad \text{for some} \quad g \in \mathcal{B}^M \cap C^k(\mathbb{R}^+).
\end{equation*}
%%%%%%%%%%%%%%%%
If $f(y) < 0$ for all $y > y_m$, we choose
\begin{equation*}
f^-(y) = f(y) - g(y), \qquad f^+(y) = g(y),  \quad \text{for some} \quad g \in \mathcal{B}^M \cap C^k(\mathbb{R}^+).
\end{equation*}
%%%
Note that the set $\mathcal{B}^M \cap C^k(\mathbb{R}^+)$ is always nonempty for each $M > 0$.
\end{remark}
}
%%%%%%%%
%%%%%%%%%%%%%%%%%%%%%%%%%%%%%%%%%%%
\begin{theorem}\label{theorem1}
Let $f: \mathbb{R}_+ \to \mathbb{R}$ be a function satisfying:
\begin{enumerate}[(i)]
\item $f(y) \in C^2(\mathbb{R}^+)$.
\item $f$ has a finite number of zeros.
\item $f(0) \geq 0$.
\end{enumerate}
Then $f$ can be represented in the form
\begin{equation}\label{eq:01}
f(y) = f^+(y) + yf^{-}(y), \quad f^+, f^- \in C^{1}(\mathbb{R}^+),  \quad f^+ \in \mathcal{F}^+, \quad f^- \in \mathcal{F}^-.
\end{equation}
\end{theorem}
\begin{proof}
The proof includes two steps as follows.\\
\textbf{Step 1.} Consider the case when $f(0) = 0$. We define
\begin{equation*}
g(y) := 
\begin{cases}
&\dfrac{f(y)}{y}, \qquad y > 0,\\
&f'(0), \qquad y = 0.
\end{cases}
\end{equation*}
It is clear that $f(y) = yg(y)$ and $g \in C^{1}(\mathbb{R}^+)$.
%\begin{equation*}
%f(y) = yg(y) \quad  \text{and} \quad g \in C^{1}(\mathbb{R}^+).
%\end{equation*}
From Lemma \ref{Lemma1}, the function $g$ admits a representation
\begin{equation*}
g(y) = g^+(y) + g^-(y), \quad g^+, g^- \in C^{1}(\mathbb{R}^+), \quad g^+ \in \mathcal{F}^+, \quad g^- \in \mathcal{F}^-.
\end{equation*}
Consequently,
\begin{equation*}
f(y) = f^+(y) + yf^-(y), \quad f^+(y) := yg^+(y),\quad f^-(y) = g^-(y).
\end{equation*}
\textbf{Step 2:} If $f(0) > 0$, we set $g(y) = f(y) - f(0)$. Then, $g(0) = 0$. From Step 1 we have that $g$ has a representation
\begin{equation*}
g(y) = g^+(y) + yg^-(y), \quad g^+, g^- \in C^{1}(\mathbb{R}^+), \quad g^+ \in \mathcal{F}^+, \quad g^- \in \mathcal{F}^-.
\end{equation*} 
Hence,
\begin{equation*}
f(y) = g(y) + f(0) = \big[f(0) + g^+(y)\big] +  yg^-(y),
\end{equation*}
where $f(0) + g^+(y) \in \mathcal{F}^+$ and $g^-(y) \in \mathcal{F}^-$. The proof is completed by combining Step 1 and Step 2.
\end{proof}
%%%%%%%%%%%%%%
{
\begin{remark}
Similarly to Remark \ref{remark1a}, the proof of Theorem \ref{theorem1} also provides a way to determine representations of the function $f$. In many cases, it is easy to find representations in the form \eqref{eq:01} for a given function $f$, for instance
\begin{equation*}
\begin{split}
&y - y^2 = y + y(-y) = (y + y^2) + 2y(-y) = (y + ye^y) + y(-y - e^y),\\
%%%%%%%%%%%%
&e^y - 1 = (e^y + y - 1) + y(-1) = (e^y + y^2 - 1) + y(-y),\\
%%%%%%%%%
&\sin(y) = \big[\sin(y) + y\big] + y(-1) = \big[\sin(y) + ye^y\big] + y(-e^y) .
\end{split}
\end{equation*}
Hence, the representations in the form \eqref{eq:01} of $f$ may not be unique. This leads to an optimal problem for the representation of the function $f$.
\end{remark}
}
\begin{remark}
In \cite{Cresson}, a representation theorem for real-valued functions  was constructed (see \cite[Theorem 10]{Cresson}); however, this theorem did not affirm the differentiability of negative and positive parts $f^+$ and $f^-$. It only proved that $f^+$ and $f^-$ are continuous. Actually, as will be seen in Lemma \ref{Maintheorem2}, the differentiability  of $f^+$ and $f^-$ is very essential in the construction of the elementary stable NSFD method.
\end{remark}
\subsection{Main results}
Throughout this paper, we always assume that:
\begin{enumerate}
\item[(A1)] The right-hand side function $f$ of Eq. \eqref{eq:1} satisfies $f(0) \geq 0$. This condition implies that $y(t) \geq 0$ for all $t \geq 0$ whenever $y_0 \geq 0$.
%%%%%%%%%%
\item[(A2)]  The equation \eqref{eq:1} has a finite number of equilibria and each of which is hyperbolic, i.e., $f'(y_*) \ne 0$ for any equilibrium point $y_*$ of \eqref{eq:1}. This condition means that the LAS of equilibria of \eqref{eq:1} can be investigated by the Lyapunov indirect method (see \cite{Allen, Stuart}).
\end{enumerate}
%%%%%%%%%%%%%%%%%%%%%%%%%%%%%%%%%%%%%%%%%%%%%%%%%%%%%%%%%%%%%%%%%%%%%%%%%%%%%%%%%%%%%%%%%%%%%%%%%%%%%%%%%%%%%%%%%%%%%%%%%%%%%%
Our main objective is to construction an NSFD method preserving simultaneously the positivity and LAS of Eq. \eqref{eq:1}.  First, we utilize Theorem \ref{theorem1} to transform the right-hand side function $f$ of Eq. \eqref{eq:1} to the form:
%%%%%%%%%%%%%%%%%%%%%%%%%%%%%%%%%%%
\begin{equation}\label{eq:2}
f(y) = f^+(y) + yf^-(y), \quad f^+, f^- \in C^1(\mathbb{R}^+), \quad f^+ \in \mathcal{F}^+, \quad f^- \in \mathcal{F}^-.
\end{equation}
Then we propose the following NSFD scheme
\begin{equation}\label{eq:3}
\dfrac{y_{n + 1} - y_n}{\varphi(h, y_n)} = f^+(y_n) + \alpha y_{n}f^-(y_n) + \beta y_{n + 1}f^-(y_n),
\end{equation}
where $\varphi: \mathbb{R}^+ \times \mathbb{R}^+ \to \mathbb{R}^+$ satisfies
\begin{equation}\label{eq:4}\
\varphi(h, y) = h + \mathcal{O}(h^2)  \quad \mbox{as} \quad h \to 0,
%\quad \varphi(h, y) > 0 \quad \text{for all} \quad (h, y) \in \mathbb{R}^+ \times \mathbb{R}^+,
\end{equation}
and $\alpha, \beta$ are real numbers satisfying 
\begin{equation}\label{eq:5}
\alpha + \beta = 1, \quad \alpha \leq 0, \quad \beta \geq 0.
\end{equation}
For convenience, the variables in the function $\varphi(h, y)$ will be omitted in some places; Furthermore, throughout the rest of this section, we always assume that the conditions \eqref{eq:4} and \eqref{eq:5} are satisfied.
%%%%%%%%%%%%%%%%%%%%%%%%%%%%%%
{
\begin{remark}
\begin{enumerate}[(i)]
\item The denominator function $\varphi$ of the NSFD method \eqref{eq:3} depends not only on  $h$ but also on $y_n$ to guarantee that it is accurate of order $2$ (see Lemma \ref{Maintheorem3}). Furthermore, it is determined by the correct asymptotic behaviour modelling.
%%%
\item Differently from most of NSFD schemes introduced in \cite{Mickens1, Mickens2, Mickens3, Mickens4, Mickens5}, the NSFD scheme \eqref{eq:3} uses a non-local approximation with the appearance of the parameters $\alpha$ and $\beta$ as weights. These parameters not only ensure the positivity (see Lemma \ref{Maintheorem1}) but also make the method more flexible. As will be seen in Subsection \ref{subsec4.1}, the NSFD method \eqref{eq:3} can become an exact scheme with suitable parameters $\alpha$ and $\beta$.
\end{enumerate}
\end{remark}
}
%%%%%%%%%%%%%%%%%%%%%%%%%%%%%%%

%%%%%%%%%%%
%%%%%%%%%%%
\begin{lemma}[Positivity of the NSFD method]\label{Maintheorem1}
Let $y_0 \geq 0$ be any initial data for Eq. \eqref{eq:1} and $\{y_n\}_{n \geq 1}$ be approximations generated by the NSFD method \eqref{eq:3}. Then we have $y_n \geq 0$ for all $n \geq 1$. In other words, the NSFD method \eqref{eq:3} preserves the positivity of Eq. \eqref{eq:1} for all finite step sizes $h > 0$.
\end{lemma}
%%%%%%%%%%%%%%%%%%%%%%%%%%%%%%
%%%%%%%%%%%%%%%%%%%%%%%%%%%%%%%%%%
\begin{proof}
The theorem is proved by mathematical induction. We first transform the scheme \eqref{eq:3} to the explicit form
\begin{equation}\label{eq:6}
y_{n + 1} = \dfrac{y_n + \varphi f^+(y_n) + \varphi \alpha y_n f^-(y_n)}{1 - \varphi \beta f^-(y_n)}
\end{equation}
Assume that $y_n \geq 0$. From
\begin{equation*}
f^+(y_n) \geq 0, \quad f^-(y_n) \leq 0, \quad \alpha \leq 0, \quad \beta \geq 0, \quad \varphi \geq 0,
\end{equation*}
we deduce that $y_{n + 1} \geq 0$. The proof is complete.
\end{proof}
%%%%%%%%%%%%%%%%%%%%%
To determine the Jacobian of the scheme \eqref{eq:3} at its equilibria conveniently, we rewrite the scheme \eqref{eq:6} in the form
\begin{equation}\label{eq:7}
y_{n + 1} = y_n + \dfrac{\varphi f(y_n)}{1 - \varphi \beta f^-(y_n)}.
\end{equation} 
Assume that the set of equilibria of Eq. \eqref{eq:1} is $\Omega = \{y_1^*, y_2^*, \ldots, y_m^*\}$. We denote by $\Omega^s$ is the set of locally asymptotically stable equilibria of Eq. \eqref{eq:1}, that is
\begin{equation*}\label{eq:8}
\Omega^s = \{y^* \in \Omega|f'(y^*) < 0\}.
\end{equation*}
%%%%%%%%%%%
%%%%%%%%%%%%%%%%%%%%%%%%%%%%%%%%%%%%%%%%%%%%%%%%%%%%%%%%%%%%%%%%%%%%%%%%%%%%%%%
%%%%%%%%%%%%%%%%%%%%%%%%%%%%%%%%%%%%%%%%%%%%%%%%%%%%%%%%%%%%
\begin{lemma}[Stability of the NSFD method]\label{Maintheorem2}
Assume that the function $\varphi(h, y)$ satisfies
\begin{equation}\label{eq:9}
\varphi(h, y^*) < \dfrac{2}{2\beta f^-(y^*) - f'(y^*)},  \qquad \forall y^* \in \Omega^s.
\end{equation}
Then the NSFD method \eqref{eq:3} is elementary stable.
\end{lemma}
%%%%%%%%%%%%%%%%%%%%%%%%%
%%%%%%%%%%%%%
\begin{proof}
It is easy to verify that equilibria of \eqref{eq:1} and \eqref{eq:7} are identical. Let $y^*$ be any equilibrium point of \eqref{eq:1}. The Jacobian of \eqref{eq:7}  at $y^*$ is given by
\begin{equation*}
J(y^*) = 1 + \dfrac{\varphi(h, y^*)f'(y^*)}{1 - \varphi(h, y^*)\beta f^-(y^*)}
\end{equation*}
We consider two following cases:\\
(i) If $f'(y^*) > 0$, then $y^*$ is an unstable equilibrium point of \eqref{eq:1}. We need to prove that $y^*$ is also an unstable equilibrium point of \eqref{eq:7}. Indeed, since
\begin{equation*}
J(y^*) = 1 + \dfrac{\varphi(h, y^*)f'(y^*)}{1 - \varphi(h, y^*)\beta f^-(y^*)} > 1,
\end{equation*}
$y^*$ is an unstable equilibrium point of \eqref{eq:7}.\\
%%%%%%%%%%%%%%%%
%%%%%%%%%%%
(ii)  If $f'(y^*) < 0$, then $y^*$ is a locally asymptotically stable equilirium point of \eqref{eq:1}. We will prove that  $y^*$ is a locally asymptotically stable equilirium point of \eqref{eq:7}. Indeed, we have
\begin{equation*}
J(y^*) = 1 + \dfrac{\varphi(h, y^*)f'(y^*)}{1 - \varphi(h, y^*)\beta f^-(y^*)} < 1.
\end{equation*}
On the other hand, $J(y^*) > -1$ if and only if
\begin{equation*}
2 > \varphi(h, y^*)\big[2 \beta f^-(y^*) - f'(y^*)\big].
\end{equation*}
This inequality is satisfied if \eqref{eq:9} holds. Hence, if \eqref{eq:9} then $|J(y^*)| < 1$, or equivalently, $y^*$ is a stable equilibrium point of \eqref{eq:7}.\\
The proof is completed.
\end{proof}
%%%%%%%%%%%%%%%%%
%%%%%%%%%%%%%%%%%%%
\begin{lemma}[The second-order NSFD method]\label{Maintheorem3}
Let $\varphi(h, y)$ be a function satisfying
\begin{equation}\label{eq:10}
\dfrac{\partial^2 \varphi}{\partial h^2}(0, y) = f'(y) - 2\beta f^-(y), \quad \forall y \geq 0.
\end{equation} 
Then the NSFD method \eqref{eq:3} is accurate of order $2$.
\end{lemma}
%%%%%%%%%%%%%%%%%%
%%%%%%%%%%%%%%%%%%
\begin{proof}
Using the Taylor expansion for $y(t)$ we obtain
\begin{equation}\label{eq:11}
y(t_{n + 1}) = y(t_n + h) = y(t_n) + hy'(t_n) + \dfrac{h^2}{2}y''(t_n) + \mathcal{O}(h^3).
\end{equation}
%%%%%%%%%%%%%%%%%
%%%%%%%%%%%%%%%%%%
Let us denote
\begin{equation*}
G(h, y) = y + \dfrac{\varphi(h, y) f(y)}{1 - \varphi(h, y) \beta f^-(y)}.
\end{equation*}
It is easy to verify that
\begin{equation*}
G(0, y) = y, \quad \dfrac{\partial G}{\partial h}(0, y) = f(y),\quad \dfrac{\partial^2G}{\partial h^2}(0, y) = \dfrac{\partial^2 \varphi}{\partial h^2}(0, y) + 2\beta f^-(y).
\end{equation*}
Consequently,
\begin{equation}\label{eq:12}
\begin{split}
y(t_n) + \dfrac{\varphi f(y(t_n))}{1 - \varphi \beta f^-(y(t_n))} &= G(0, y(t_n)) + h\dfrac{\partial G}{\partial h}(0, y(t_n)) + \dfrac{h^2}{2}\dfrac{\partial^2G}{\partial h^2}(0, y(t_n)) + \mathcal{O}(h^3)\\
&= y(t_n) + hf(y(t_n)) + \dfrac{h^2}{2}\bigg[ \dfrac{\partial^2 \varphi}{\partial h^2}(0, y(t_n)) + 2\beta f^-(y(t_n))\bigg] + \mathcal{O}(h^3).
\end{split}
\end{equation}
From \eqref{eq:11} and \eqref{eq:12} we deduce that if \eqref{eq:10} holds then
\begin{equation*}
y(t_{n + 1}) - \bigg[y(t_n) + \dfrac{\varphi f(y(t_n))}{1 - \varphi \beta f^-(y(t_n))}\bigg] = \mathcal{O}(h^3),
\end{equation*}
which means that the NSFD scheme is accurate of order 2.
\end{proof}
%%%%%%%%
%%%%%%%%%%
From Lemmas \ref{Maintheorem1}, \ref{Maintheorem2} and \ref{Maintheorem3}, we obtain
\begin{theorem}\label{Maintheorem}
The NSFD method \eqref{eq:3} is positive, elementary stable and accurate of order $2$ if
\begin{equation}\label{eq:16}
\begin{split}
(H1) \qquad &\varphi(h, y) = h + \mathcal{O}(h^2)  \quad \mbox{as} \quad h \to 0, \quad \varphi(h, y) \geq 0 \quad \text{for all} \quad (h, y) \in \mathbb{R}^+ \times \mathbb{R}^+,\\
%%%%%%%%%%%%%%%%%%%%%%
(H2) \qquad &\varphi(h, y^*) < \dfrac{2}{2\beta f^-(y^*) - f'(y^*)},  \qquad \forall y^* \in \Omega^s,\\
%%%%%%%%%%%%%%%%
(H3) \qquad &\dfrac{\partial^2 \varphi}{\partial h^2}(0, y) = f'(y) - 2\beta f^-(y), \quad \forall y \geq 0.\\
(H4) \qquad &\alpha + \beta = 1, \quad \alpha \leq 0, \quad \beta \geq 0.
\end{split}
\end{equation} 
\end{theorem}
%%%%%%%%%%%%%%%%%%%%%%%%%%%%%%%%%%%%%%%%%%%%%%%%%%%%%
Next, we introduce a function $\varphi(h)$ satisfying the conditions $(H1)-(H3)$. Let us denote 
\begin{equation*}
\lambda(y) := -f'(y) + 2\beta f^-(y),
\end{equation*}
\begin{equation}\label{eq:17}
\varphi(h, y) =
\begin{cases}
\dfrac{1 - e^{-h\lambda(y)}}{\lambda(y)}, \quad &\lambda(y) \ne 0,\\
h, \qquad  &\lambda(y) = 0.
\end{cases}
\end{equation}
\begin{proposition}\label{Proposition1}
The function $\varphi(h, y)$ given by \eqref{eq:17} satisfies the conditions $(H1)-(H3)$ in Theorem \ref{Maintheorem}.
\end{proposition}
\begin{proof}
It is easy to verify that $\varphi(h, y)$ satisfies $(H1)$ and $(H3)$. Since $\lambda(y^*) > 0$ for all $y^* \in \Omega^s$, we have
\begin{equation*}
\varphi(h, y^*) < \dfrac{1}{\lambda(y^*)} = \dfrac{1}{-f'(y^*) + 2\beta f^-(y^*)} < \dfrac{2}{-f'(y^*) + 2\beta f^-(y^*)}, 
\end{equation*}
which implies that $(H2)$ is also satisfied. The proof is completed.
\end{proof}
\begin{remark}
\begin{enumerate}[(i)]
\item Lemma 2 in \cite{Gupta} provides an important result for selecting the denominator functions satisfying the conditions $(H1)-(H3)$ in Theorem \ref{Maintheorem}.
\item  The convergence of the NSFD scheme \eqref{eq:3} can be established similarly to the results constructed in \cite[Section 6]{Cresson}.
\end{enumerate}
%%%%%%%%%%%%%%%%%%%%%%%%%%%%%%%%%%%%%%%%%%%%%%%%%%%%%%%%%%%%%%%%%%%%%%%%%%%%%
\end{remark}
{
\section{Some applications of the NSFD method}\label{sec4}
In this section, we present some numerical examples and applications of the NSFD method \eqref{eq:3}. {  In all the following examples, the denominator functions $\varphi$ for the NSFD method \eqref{eq:3} are determined by using Proposition \ref{Proposition1}.}
\subsection{The logistic differential equation}\label{subsec4.1}
We consider the well-known logistic equation (see \cite{Allen})
\begin{equation}\label{eq:18}
\dfrac{dy}{dt} = 2\bigg(y - \dfrac{y^2}{2}\bigg) := f(y), \quad 0 \leq t \leq 1, \quad y(0) = 0.5.
\end{equation}
The exact solution of this equation is given by 
\begin{equation*}
y(t) = \dfrac{1}{0.5 + 1.5e^{-2t}}.
\end{equation*}
The equation has two equilibria $y_0^* = 0$ (unstable) and $y_1^* = 2$ (stable).\par
{
From the NSFD method \eqref{eq:3}, we obtain the following second-order NSFD schemes for the Eq. \eqref{eq:18}, which correspond to three different representations of the right-hand side function $f(y) = 2y - y^2$ given in Table \ref{Table1new}:
%%%
\begin{table}[H]
\begin{center}
\caption{Second-order NSFD schemes for E.q \eqref{eq:18}.}\label{Table1new}
\begin{tabular}{ccccccccccccc}
\hline
Scheme&$f^+$&$f^-$&$\beta$&Denominator function&Remark\\
\hline
&&&&&&\\
sNSFD1&$2y$&$-y$&$1.25$&$\varphi(h, y) = \dfrac{1 - e^{(2 + 0.5y)h}}{2 + 0.5y}$&Second-order scheme\\
&&&&&&\\
%%%%
sNSFD2&$2y + y^2$&$-2y$&$1.25$&$\quad \varphi(h, y) = \dfrac{1 - e^{-h(2 + 3y)}}{2 + 3y}$&Second-order scheme\\
&&&&&&\\
sNSFD3&$2y$&$-y$&$1$&$\varphi(h, y) = \dfrac{1 - e^{-2h}}{2}$&Exact NSFD scheme\\
&&&&&&\\
\hline
\end{tabular}
\end{center}
\end{table}
We now observe absolute errors and convergence rates of the schemes sNSFD1 and sNSFD2 at the end of the integration interval, namely, $T = 1$. Numerical results are presented in Table \ref{Tabl1}. In this table,
\begin{equation*}
error = |y(t_N) - y_N|, \quad rate =: log_{(h_1/h_2)}\bigg(\dfrac{error(h_1)}{error(h_2)}\bigg).
\end{equation*}
%%%%%%%%%%%%%%%%%%%%%%%%%%%%%%%%%%%%%%%%%%%%
\begin{table}[H]
\caption{Absolute errors and rates of the NSFD schemes for Eq. \eqref{eq:18}.}\label{Tabl1}
\begin{center}
\begin{tabular}{cccccccc}
\hline
Step h&sNSFD1 error&rate&sNSFD2 error&rate&Wood \& Kojourharov's scheme \eqref{eq:19}&rate\\
\hline
$10^{-1}$&$0.0014$&&$0.0127$&&$0.0470$&\\
%%%%%%%%%%%%%%%%%%%%%%
$10^{-2}$&$1.4678e-005$&$1.9795$&$1.3823e-004$&$1.9632$&$0.0045$&$1.0189$&\\
%%%%%%%%%%%%%%%%%%%%
$10^{-3}$&$ 1.4749e-007$&$1.9979$&$1.3910e-006$&$1.9973$&$ 4.4841e-004$&$1.0015$&\\
%%%%%%%%%%%%%%%%%%%%
$10^{-4}$&$1.4756e-009$&$1.9998$&$1.3918e-008$&$1.9998$&$4.4820e-005$&$1.0002 $&\\
%%%%%%%%%%%%%%%%%%%%
$10^{-5}$&$1.4742e-011$&$2.0004$&$1.3917e-010$&$2.0000$&$ 4.4818e-006$&$1.0000$&\\
\hline
\end{tabular}
\end{center}
\end{table}
It is clear that the schemes sNSFD1 and sNSFD2 are accurate of order 2. However, the scheme sNSFD1 provides better errors than the scheme sNSFD2. This means that the representations of the function $f$ affect the errors of the NSFD schemes. This observation leads to an optimal problem for the representation of the function $f$.\par
%%%
It is important to note that in the case $\beta = 1$, the NSFD method \eqref{eq:3} is accurate of order $2$ for Eq. \eqref{eq:18} if
\begin{equation*}
\dfrac{\partial^2\varphi}{\partial h^2}(0, y) = f'(y) - 2\beta f^-(y) \equiv 2.
\end{equation*}
In this case, the function $\varphi(h) = (1 - e^{-2h})/{2}$ satisfies this condition. For this function, we do not need to update the value of the denominator function in each iterative step; consequently, computational volume and computation time are reduced. This show the role of the parameter $\beta$.  Actually, in the case when $\varphi(h) = (1 - e^{-2h})/{2}$ we obtain an \textit{exact scheme} for the logistic equation \cite{Mickens1, Mickens2, Mickens3, Mickens4, Mickens5}.\par
}
%%%%%%
We now use compare the second NSFD schemes with positive and elementary stable NSFD (PSENSFD) method constructed by Wood and Kojouharov in \cite{Wood}. Note that the Wood and Kojourharov's scheme is given by
\begin{equation}\label{eq:19}
y_{n + 1}
=
\begin{cases}
&y_n + \phi(h)\big(2y_n - y_n^2\big), \quad \text{if} \quad 2y_n - y_n^2\geq 0,\\
%%%%%%%%%%%%%%%%%%
&\dfrac{y_n^2}{y_n - \phi(h)\big(2y_n - y_n^2\big)}, \quad \text{if} \quad 2y_n - y_n^2 < 0,
\end{cases}
\end{equation}
where $\varphi(h)$ satisfies $\phi(h) < 1$ for all $h > 0$. Therefore, we take $\phi(h) = 1 - e^{-h}$. The errors and convergence rates of the scheme \eqref{eq:19} are also presented in Table \ref{Tabl1}. It is clear that it is only accurate of order $1$. However, it requires the computation of the sign of the right-hand side function at each iterative step. \par
%

%%%%%%%%%%
%%%%%%%%%%%%
To end this subsection, we consider dynamics of numerical solutions generated by the standard Euler scheme, the standard second-order Runge-Kutta (RK2) scheme and the NSFD method \eqref{eq:3} (scheme sNSFD1) with the step size $h = 1.25$. The numerical approximations are depicted in Figures \ref{Fig:1} and \ref{Fig:2}. It is easy to observe that the standard Euler scheme generated the numerical approximation oscillates around the equilibrium position, the approximation generated by the RK2 scheme converges to a spurious  equilibrium point; however, the NSFD scheme preserves the dynamics of the logistic equation. This result is completely consistent with recognized results on NSFD schemes for differential equations \cite{Anguelov, Kojouharov, Mickens1, Mickens2, Mickens3, Mickens3, Mickens4, Mickens5, Patidar1, Patidar2, Wood}. On the other hand, it is easy to observe that the NSFD method \eqref{eq:3} is more  accurate than the PESNSFD method.
%%%%%%%%%%%
\begin{figure}[H]
\centering
\includegraphics[height=10cm,width=15cm]{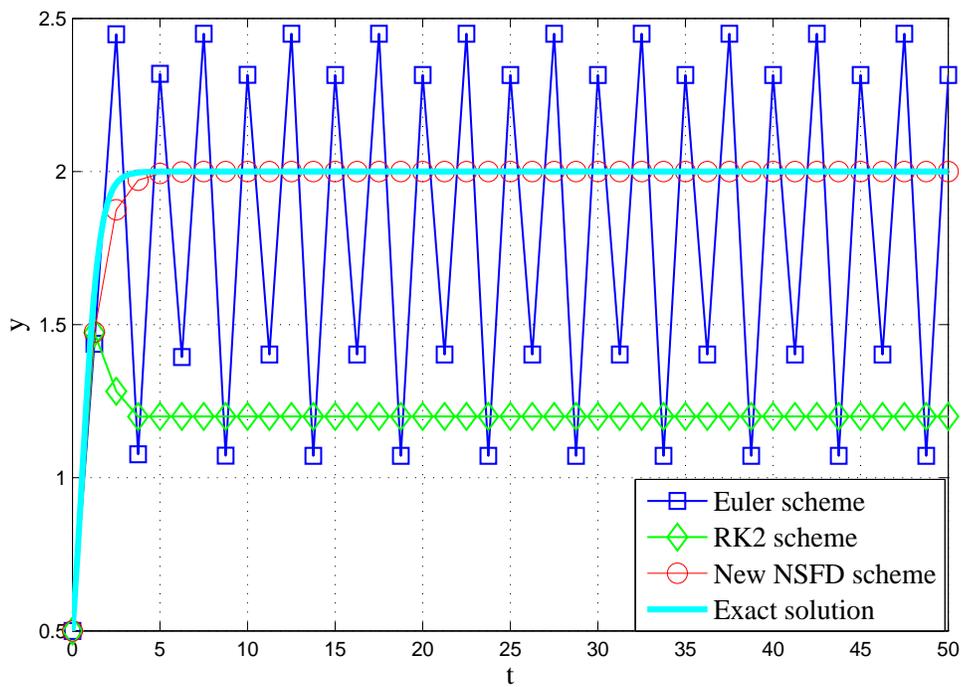}
\caption{The approximations generated by the Euler scheme, the RK2 scheme and the NSFD scheme with $h = 1.25$.}\label{Fig:1}
\end{figure}
%%%%%%%%%%%%%%%%%%%%%%%%%
\begin{figure}[H]
\centering
\includegraphics[height=10cm,width=15cm]{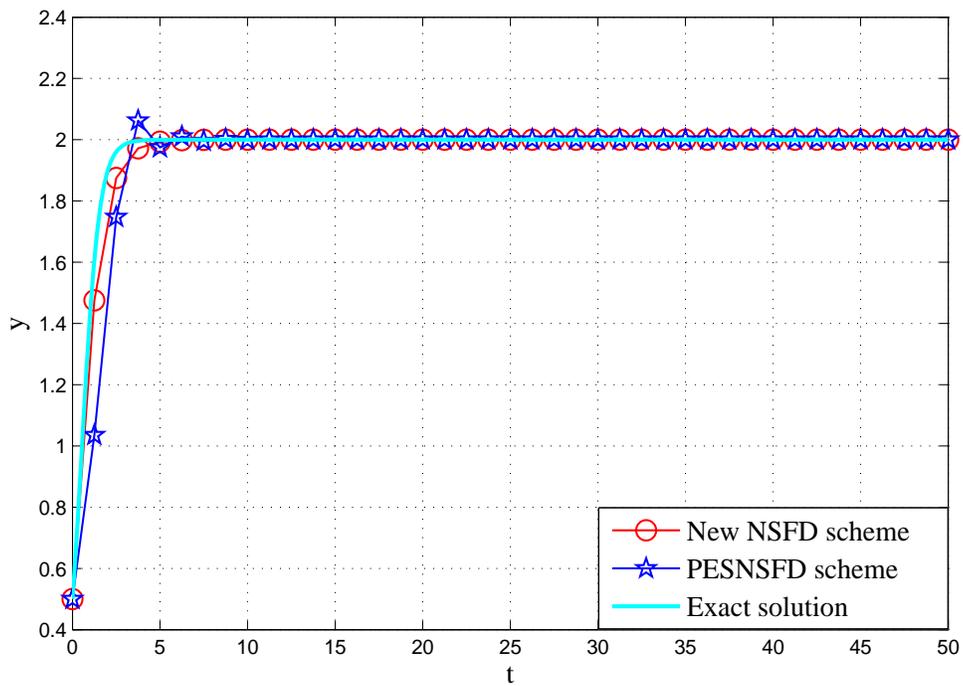}
\caption{The approximations generated by the new NSFD method and PENSFD scheme (scheme \eqref{eq:19}) with $h = 1.25$.}\label{Fig:2}
\end{figure}
%%%%%%%%%%%%%%%%%%%%%%%%
%%%%%%%%%%%%%%%%%%%%%%%%%
\subsection{The cubic differential equation}
Consider the cubic differential equation
\begin{equation}\label{eq:20}
\dfrac{dy}{dt} = y(1 - y^2), \quad y(0) \geq 0. 
\end{equation}
In $\mathbb{R}^+$, the cubic equation \eqref{eq:20} has two equilibrium point $y^* = 0$ (unstable) and $y^* = 1$ (stable).\par
%%%%%%%%%%%%%%%%%%%%%%%%%%%%%%%%
Now, Theorem \ref{Maintheorem} will be applied to obtain a positive and elementary stable NSFD scheme of second order for Eq. \eqref{eq:20}. For this purpose, we choose
\begin{equation*}
f^+(y) = y, \quad f^-(y) = y^2, \quad \beta = \dfrac{3}{2}, \quad \alpha = -\dfrac{1}{2}.
\end{equation*}
Then we obtain
\begin{equation*}
\dfrac{y_{n + 1} - y_n}{\varphi} = y_n + \dfrac{1}{2}(y_n)^3 - \dfrac{3}{2}(y_n)^2y_{n + 1},
\end{equation*}
or equivalently,
\begin{equation}\label{eq:21a}
y_{n + 1} = \dfrac{(2 + \varphi)y_n + \varphi (y_n)^3}{2 + 3\varphi (y_n)^2}.
\end{equation}
By Theorem \ref{Maintheorem}, the scheme \eqref{eq:21a} is positive, elementary stable and accurate of order $2$ if
\begin{equation*}
\dfrac{\partial^2 \varphi}{\partial h^2}(0, y) = f'(y) - 2\beta f^-(y) = 1.
\end{equation*}
For this condition, the function $1 - e^{-h}$ is most suitable; consequently, we obtain
\begin{equation}\label{eq:21b}
y_{n + 1} = \dfrac{(2 + \varphi)y_n + \varphi (y_n)^3}{2 + 3\varphi (y_n)^2}, \quad \varphi(h) = 1 - e^{-h}.
\end{equation}
%%%%%%%%%%%%%%%%%%%%%%%%%%%%%%%%%%%%%%%%%%%%%%%%%%
In \cite{Mickens6}, Mickens proposed the following NSFD scheme that incorporates the maximum symmetry in modeling the nonlinear term:
\begin{equation}\label{eq:21}
y_{n + 1} = \Bigg[\dfrac{(2 + \phi) + \phi(y_n)^2}{(2 - \phi) + 3\phi(y_n)^2}\Bigg]y_n ,\quad \phi = \dfrac{1 - e^{-2h}}{2}.
\end{equation}
%%%%%%%%%%%%%%%%%%%%%%%%%%%%%%%%%%%%%%%%%%%%%%%%%%%%%
%%%%%%%%%%%%%%%%%%%%%%%%%%%%%%%%%%%%%%%%%%%%%%%%%%
Clearly, the scheme \eqref{eq:21b} is different from \eqref{eq:21}. However, it is easy to show that the NSFD scheme \eqref{eq:21} is a particular case of the NSFD method \eqref{eq:3}. Indeed, for the NSFD scheme \eqref{eq:3} we choose
\begin{equation*}
f^+(y) = \dfrac{1}{2}y + (y)^3, \quad f^-(y) = \dfrac{1}{2} - \dfrac{3}{2}y^2, \quad \alpha = 0, \quad \beta = 1.
\end{equation*}
Note that in this form $f^- \notin \mathcal{F}^-$. Then the NSFD method \eqref{eq:3} for Eq. \eqref{eq:20} becomes
\begin{equation*}
\dfrac{y_{n + 1} - y_n}{\varphi(h)} = \dfrac{1}{2}y_n + \dfrac{1}{2} (y_n)^3 + \dfrac{1}{2}y_{n + 1} - \dfrac{3}{2}  (y_n)^2 y_{n + 1},
\end{equation*}
or equivalently,
\begin{equation}\label{eq:22}
y_{n + 1} = \Bigg[\dfrac{(2 + \varphi) + \varphi(y_n)^2}{(2 - \varphi) + 3\varphi(y_n)^2}\Bigg]y_n.
\end{equation}
It is clear that the scheme \eqref{eq:22} will coincident with \eqref{eq:21} if $\varphi = ({1 - e^{-2h})}/{2}$. Since $f^-(y) = (1/2) - (3/2)y^2 \notin \mathcal{F}^-$, the scheme \eqref{eq:22} is not positive for any function $\varphi(h) > 0$. A simple condition for the scheme to be positive is $\varphi(h) < 1/2$ for all $h > 0$. Fortunately, the function $\varphi = ({1 - e^{-2h})}/{2}$ satisfies this condition; therefore, the scheme \eqref{eq:21} is positive. Importantly, the function $\varphi = ({1 - e^{-2h})}/{2}$ satisfies Theorem \ref{Maintheorem}; consequently, the scheme \eqref{eq:21} is also of second order accuracy.\par
%%%%%%%%%%%%%%%%%%%%
%%%%%%%%%%%%%%%%%%%%%%%%
The results presented in this subsection and Subsection \ref{subsec4.1} can be generalized as follows.
\begin{proposition}
The ODE
\begin{equation}
\dfrac{dy}{dt} = ay - by^m, \quad a, b > 0,  \quad m \geq 2, \quad y(0) \geq 0
\end{equation}
has a positive and elementary stable NSFD scheme of second order accuracy defined by
\begin{equation}\label{eq:23a}
\dfrac{y_{n + 1} - y_n}{\varphi(h)} = ay_n - b\bigg(1 - \dfrac{m}{2}\bigg)(y_n)^m - b\dfrac{m}{2}{(y_n)}^{m-1}y_{n + 1}, \quad \varphi(h) = \dfrac{1 - e^{-ah}}{a}.
\end{equation}
\end{proposition}
\subsection{The  modified Monod equation}
In \cite{Mickens6}, Mickens considered the modified Monod equation
\begin{equation}\label{eq:23}
\dfrac{dy}{dt} = \dfrac{(\lambda - 1)y - (\lambda+1)y^2}{1 + y}, \quad \lambda > 1,
\end{equation}
and proposed the following NSFD scheme
\begin{equation}\label{eq:24}
\dfrac{y_{n + 1} - y_n}{\phi} = \Bigg[\dfrac{(\lambda - 1) - (1 + \lambda)y_{n + 1}}{1 + y_n}\Bigg]y_n, \quad \phi = \dfrac{1 - e^{-Rh}}{R}, \quad R = \lambda - 1.
\end{equation}
Clearly, the scheme \eqref{eq:24} is a particular case of the NSFD method \eqref{eq:4} with
\begin{equation*}
f^+(y) = \dfrac{(\lambda - 1)y}{1 + y}, \,\,  f^-(y) = \dfrac{(\lambda+1)y}{1 + y}, \,\, \beta = 1, \,\, \varphi(h, y) = \dfrac{1 - e^{-Rh}}{R}, \,\, R = \lambda - 1.
\end{equation*}
Although the scheme \eqref{eq:24} is positive and elementary stable, it is only of first order accuracy. However, thanks to Theorem \ref{Maintheorem} it is easy to determine conditions for this scheme to be accurate of order $2$. More clearly, the scheme \eqref{eq:24} is of second-order accuracy if
\begin{equation*}
\dfrac{\partial^2 \varphi}{\partial h^2}(0, y) = -\dfrac{(\lambda + 3) + 4(\lambda + 1)y + 3(\lambda + 1)y^2}{(1 + y)^2}.
\end{equation*}
A denominator function satisfying this above condition is
\begin{equation*}
\varphi(h, y) = \dfrac{1 - e^{-R(y)h}}{R(y)}, \quad R(y) = \dfrac{(\lambda + 1) + 4(\lambda + 1)y + 3(\lambda + 1)y^2}{(1 + y)^2}.
\end{equation*}
This function ensures that the scheme \eqref{eq:24} is of second order accuracy.
\subsection{The sine equation}
In \cite{Mickens6}, Mickens considered the since equation
\begin{equation}\label{eq:25a}
\dfrac{dy}{dt} = \sin(\pi y)
\end{equation}
and formulated the following NSFD scheme 
\begin{equation}\label{eq:25}
\dfrac{y_{n + 1} - y_n}{\phi} = sin(\pi y_n), \quad \phi = \dfrac{1 - e^{-\pi h}}{\pi}.
\end{equation}
It is easy to verify that the scheme \eqref{eq:25} is of first order accuracy. We now use Theorem \ref{Maintheorem} to obtain a second-order NSFD scheme for the sine equation. For this purpose, we rewrite Eq. \eqref{eq:25a} in the form
\begin{equation*}
\dfrac{dy}{dt} = \sin(\pi y) = f^+(y) + yf^-(y) = \big[sin(\pi y) + \pi y\big] + y(-\pi).
\end{equation*}
Note that $y + sin(y) \geq 0$ for all $y \geq 0$. Applying the NSFD method \eqref{eq:3} with $\beta = 1$ we obtain
\begin{equation}\label{eq:26}
y_{n + 1} = \dfrac{y_n + \varphi \big[sin(\pi y_n) + \pi y_n\big]}{1 + \varphi \pi}.
\end{equation}
The scheme is elementary stable and positive, and it is accurate of order $2$ if
\begin{equation*}
\dfrac{\partial^2 \varphi}{\partial h^2}(0, y) = \pi\cos(\pi y) - 2\pi.
\end{equation*}
Hence, the function
\begin{equation*}
\varphi(h, y) = \dfrac{1 - e^{-\lambda(y)h}}{\lambda(y)}, \quad \lambda(y) = \pi\cos(\pi y) - 2\pi
\end{equation*}
guarantees that the NSFD scheme \eqref{eq:26} is positive, elementary stable and accurate of order $2$.
{
\subsection{An extension of the NSFD method for systems of ODEs}
In this subsection, we generalize the NSFD method \eqref{eq:3} for systems of ODEs. For this purpose, we first consider the following general class of two-dimensional differential equations including several models in population dynamics
\begin{equation}\label{eq:29}
\begin{split}
&\dfrac{dx}{dt} = x\big(f_+(x, y) - f_-(x, y)\big) := f(x, y),\quad x(0) = x_0 \geq 0,\\
&\dfrac{dy}{dt} = y\big(g_+(x, y) - g_-(x, y)\big) := g(x, y),\quad y(0) = y_0 \geq 0,
\end{split}
\end{equation}
where $f_+, f_-$ and $g_+, g_-$ are positive for all $(x, y) \in \mathbb{R}_+\times\mathbb{R}_+$ and of class $\mathcal{C}^1$. In \cite{Cresson}, Cresson and Pierret constructed a nonstandard finite difference method preserving dynamical properties including positivity and stability of the model \eqref{eq:29}. It was proved that the constructed NSFD method is convergent and of order one (see \cite[Theorem 5.2]{Cresson}).\par
We now utilize the approach is this work to construct a second-order NSFD method for the model \eqref{eq:29}. For this reason, we propose the following NSFD method
%%%
\begin{equation}\label{eq:30}
\begin{split}
&\dfrac{x_{n + 1} - x_n}{\varphi_1(x_n, y_n, h)} = x_nf_+(x_n, y_n) - \alpha_1x_{n}f_-(x_n, y_n) - \beta_1x_{n + 1}f_-(x_n, y_n),\\
&\dfrac{y_{n + 1} - y_n}{\varphi_2(x_n, y_n, h)}  = y_ng_+(x_n, y_n) - \alpha_2y_{n}g_-(x_n, y_n) - \beta_2y_{n+1}g_-(x_n, y_n),
\end{split}
\end{equation}
where $\varphi_1$ and $\varphi_2$ are denominator functions determined by the correct asymptotic behaviour modelling and
\begin{equation}\label{eq:32new}
\begin{split}
&\alpha_1 + \beta_1 = 1, \quad \alpha_1 \leq 0, \quad \beta_1 \geq 0,\\
&\alpha_2 + \beta_2 = 1, \quad \alpha_2 \leq 0, \quad \beta_2 \geq 0.
\end{split}
\end{equation}
Note that when $\varphi_1$ and $\varphi_2$ do not depend on $x_n$ and $y_n$, we will obtain the NSFD method constructed in \cite{Cresson}.\par
%%%%%%%%%
%%%%%%%%%%
The method \eqref{eq:30} can be rewritten in the explicit form
\begin{equation}\label{eq:31}
\begin{split}
&x_{n + 1} = \dfrac{x_n + \varphi_1x_nf_+(x_n, y_n) - \varphi_1\alpha_1 x_nf_-(x_n, y_n)}{1 + \varphi_1\beta_1f_-(x_n, y_n)},\\
&y_{n + 1} = \dfrac{y_n + \varphi_2y_ng_+(x_n, y_n) - \varphi_2\alpha_2y_ng_-(x_n, y_n)}{1 + \varphi_2\beta_2g_-(x_n, y_n)}.
\end{split}
\end{equation}
Therefore, the NSFD method \eqref{eq:31} is positive for all $h > 0$ if the condition \eqref{eq:32new} holds. Assume that \eqref{eq:32new} is satisfied. Similarly to Theorem \ref{Maintheorem3}, we obtain:
\begin{theorem}\label{maintheorem5}
The NSFD method \eqref{eq:31} is positive and accurate of order $2$ if
\begin{equation}\label{eq:32}
\begin{split}
&\dfrac{\partial^2 \varphi_1}{\partial h^2}(x, y, 0) = \dfrac{\partial f}{\partial x}(x, y)f(x, y) + \dfrac{\partial f}{\partial y}(x, y)g(x, y) - 2 \beta_1 f_-(x, y),\\
&\dfrac{\partial^2 \varphi_2}{\partial h^2}(x, y, 0) = \dfrac{\partial g}{\partial x}(x, y)f(x, y) + \dfrac{\partial g}{\partial y}(x, y)g(x, y) - 2 \beta_2 g_-(x, y).
\end{split}
\end{equation}
\end{theorem}
Using the approach in \cite{Cresson}, thanks to the Lyapunov indirect method \cite{Allen, Stuart}, we can determine a stability thresholds $h^* = (h_1^*, h_2^*)$ for the NSFD method \eqref{eq:31}, i.e., the method \eqref{eq:31} preserves the local asymptotic stability of the model \eqref{eq:29} if
\begin{equation}\label{eq:33}
0 < \varphi_1 < h_1^*, \quad 0 < \varphi_2 < h_2^*.
\end{equation}
Hence, we only need to choose denominator functions satisfying the conditions \eqref{eq:32} and \eqref{eq:33}. Such denominator functions can be determined based on Proposition \ref{Proposition1}.\par
Generally, we now consider the following general class of dynamical systems
\begin{equation}\label{eq:34}
\dfrac{dy_i}{dt} = f_i(y_1, y_2, \ldots, y_n), \quad y_i(0) = y_{i}^0 \geq 0, \quad i = 1, 2, \ldots, n,
\end{equation}
where the right-hand function $f$ satisfies appropriate conditions, which guarantee that solutions of \eqref{eq:34} exist, are unique and positive for $t > 0$.\par 
Suppose that the functions $f_i$ ($i = 1, 2, \ldots, n$) admit a representation of the form
\begin{equation}\label{eq:35}
f_i(y) = f_i^+(y) + y_if_i^-(y),
\end{equation} 
where $f^+, f^-$ are of $\mathcal{C}^1$ class, $f^+$ is positive and $f^-$ is negative for all $y = (y_1, y_2, \ldots, y_n) \in \mathbb{R}^n_+$. Then, we propose the following NSFD method
\begin{equation}\label{eq:36}
\dfrac{y_i^{n + 1} - y_i^n}{\varphi_i(y^n, h)} = f_i^+(y^n) + \alpha_i y^nf_i^-(y^n) + \beta_i y_i^{n+1}f^-(y^n),
\end{equation}
with
\begin{equation*}
\alpha_i + \beta_i = 1, \quad \alpha_i \leq 0, \quad \beta_i \geq 0.
\end{equation*}
Using the approach used in Section \ref{sec3}, we can determine conditions for the function $\varphi_i$ such that the NSFD scheme \eqref{eq:36} is positive, elementary stable and of second-order accuracy.\par
In general, it is not an easy task to verify whether the function $f$ admits a representation of the form \eqref{eq:35}. However, in some cases, it is easy to find a such representation of the function $f$, for example, for the classical Lotka-Volterra system \cite{Allen}, we can write
\begin{equation}\label{eq:38}
\begin{split}
&\dfrac{dx}{dt} = ax - bxy = ax + bx(-y), \quad a, b > 0,\\
&\dfrac{dy}{dt} = -cy + exy = exy + y(-c), \quad c, e > 0
\end{split}
\end{equation} 
or for the classical SIRS epidemic model \cite{Allen}, we have
\begin{equation}
\begin{split}
&\dfrac{dS}{dt} = -\dfrac{\beta}{N}SI + \mu R = \mu R + S\bigg(-\dfrac{\beta}{N}I\bigg), \quad \beta, \mu > 0,\\
&\dfrac{dI}{dt} = \dfrac{\beta}{N}SI - \gamma I = \dfrac{\beta}{N}SI  + I(-\gamma),\quad \gamma > 0,\\
&\dfrac{dR}{dt} = \gamma I - \mu R = \gamma I + R(-\mu).
\end{split}
\end{equation}
}
\section{Conclusions and remarks}\label{sec5}
In this work, we have constructed the NSFD method that is not only positive and elementary stable but also accurate of order $2$. The theoretical assertions are supported by a set of numerical simulations. 
The obtained results not only resolve the contradiction between the dynamic consistency and high-order accuracy of NSFD schemes but also improve the well-known results constructed in \cite{Gupta, Kojouharov, Mickens6, Wood}. Furthermore, the constructed method is easy to be implemented and can be used to solve a larger class of differential equations arising in both theory and applications. \par
%%%%%%%%%%%%%%%%%%%%%%%%%%%%%%%%%%%%%%%%%%%%%%
%%%%%%%%%%%%%%%%
%%%%%%%%%%%%%%%%%%%%%%%%%%%%%%%%%%%%%%%%%%%%%%%%%%
Our future research directions include the construction of high-order NSFD methods that are not only elementary stable and positive but also preserve other important qualitative  features of solutions of differential equations, such as global asymptotic stability and periodicity. Also, the construction of dynamically consistent NSFD methods having high-order accuracy for general multi-dimensional autonomous and non-autonomous dynamical systems governed by ODEs and fractional-order ODEs will be studied.


\section*{References}
\begin{thebibliography}{00}
%A
\bibitem{Allen}
L. J. S. Allen, An Introduction to Mathematical Biology, Prentice Hall, 2007.
\bibitem{Anguelov}
R. Anguelov, J. M. -S. Lubuma, Contributions to the mathematics of the nonstandard finite difference method and Applications, Numerical Methods for Partial Differential Equations 17(2001) 518-543.
\bibitem{Anguelov1}
R. Anguelov, J. M. -S. Lubuma, Nonstandard finite difference method by nonlocal approximation, Mathematics and Computers in Simulation 61(2003) 465-475.
%C
\bibitem{Chen-Charpentier}
B. M. Chen-Charpentier, D. T. Dimitrov, H. V. Kojouharov, Combined nonstandard numerical methods for ODEs with polynomial right-hand sides, Mathematics and Computers in Simulation  73(2006) 105-113.
\bibitem{Cresson}
J. Cresson, A. Szafranska, Discrete and continuous fractional persistence problems – the positivity property and applications, Communications in Nonlinear Science and Numerical Simulation 44(2017) 424-448.
%G
\bibitem{Gupta}
M. Gupta, J. M. Slezak, F. Alalhareth, S. Roy, H. V. Kojouharov, Second-order Nonstandard Explicit Euler Method, AIP Conference Proceedings 2302(2020) 110003.
%D
\bibitem{DH1}
Quang A Dang, Manh Tuan Hoang, Positive and elementary stable explicit nonstandard Runge-Kutta methods for a class of autonomous dynamical systems, International Journal of Computer Mathematics 97 (2020) 2036-2054.
%%%%%%%%%%
\bibitem{Dimitrov}
D. T. Dimitrov, H. V. Kojouharov, Nonstandard finite-difference schemes for general two-dimensional autonomous dynamical systems, Applied  Mathematics Letters 18(2005) 769-774.
\bibitem{Dimitrov1}
D. T. Dimitrov, H. V. Kojouharov, Stability-preserving Finite Difference Methods for General Multi-dimensional Autonomous Dynamical Systems, International Journal of Numerical Analysis and Modeling 4(2007) 280-290.
%H
\bibitem{Horvath}
Z. Horv\'{a}th, Positivity of Runge-Kutta and diagonally split Runge-Kutta methods, Applied Numerical Mathematics 28(1998) 309-326.
%K
\bibitem{Kojouharov}
H. V. Kojouharov, S. Roy, M. Gupta, F. Alalhareth, J. M. Slezak, A second-order modified nonstandard theta method for one-dimensional autonomous differential equations, Applied Mathematics Letters
112(2021) 106775.
%M
\bibitem{MVaquero1} 
J. Mart\'{i}n-Vaquero, A.  Mart\'{i}n del Rey, A. H.  Encinas,  J. D.  Hern\'{a}ndez Guill\'{e}n,  A.  Queiruga-Dios and G. Rodr\'{i}guez S\'{a}nchez, Higher-order nonstandard finite difference schemes for a MSEIR model for a malware propagation, Journal of Computational and Applied Mathematics {317}(2017) 146-156.
%%%%%%%%%%%%%%
\bibitem{MVaquero2} 
J. Mart\'{i}n-Vaquero, A.  Mart\'{i}n del Rey, A. H.  Encinas,  J. D.  Hern\'{a}ndez Guill\'{e}n,  A.  Queiruga-Dios and G. Rodr\'{i}guez S\'{a}nchez, {Variable step length algorithms with high-order extrapolated non-standard finite difference schemes for a SEIR model}, {Journal of Computational and Applied Mathematics} {330}(2018) 848-854.
%%%%%%%%%%%%%%%%%%%%%%%%%%%
\bibitem{Mickens1}
R. E. Mickens, Nonstandard Finite Difference Models of Differential Equations, World Scientific, 1993.
\bibitem{Mickens2}
R. E. Mickens, Applications of Nonstandard Finite Difference Schemes, World Scientific, 2000.
\bibitem{Mickens3}
R. E. Mickens, Advances in the Applications of Nonstandard Finite Difference Schemes, World Scientific, 2005.
\bibitem{Mickens4}
R. E. Mickens, Nonstandard Finite Difference Schemes for Differential Equations, Journal of Difference Equations and Applications 8(2002) 823-847.
\bibitem{Mickens5}
R. E. Mickens, Nonstandard Finite Difference Schemes: Methodology and Applications,  World Scientific, 2020.
\bibitem{Mickens6}
R. E. Mickens, Discretizations of nonlinear differential equations using explicit nonstandard methods, Journal of Computational and Applied Mathematics 110(1999) 181-185.
%P
\bibitem{Patidar1}
K. C. Patidar, On the use of nonstandard finite difference methods, Journal of Difference Equations and Applications 11(2005) 735-758.
\bibitem{Patidar2}
K. C. Patidar, Nonstandard finite difference methods: recent trends and further developments, Journal of Difference Equations and Applications 22(2016) 817-849.
%S
\bibitem{Stuart}
A. Stuart, A. R. Humphries, Dynamical systems and numerical analysis, Cambridge University Press, 1998.
%W
\bibitem{Wood}
D. Wood, H. V. Kojouharov, A class of nonstandard numerical methods for autonomous dynamical systems, Applied Mathematics Letters 50(2015) 78-82.
\end{thebibliography}
\end{document}